\documentclass[a4paper]{scrartcl}
\usepackage[latin1]{inputenc}
\usepackage{amsmath}
\usepackage{bbm}
\usepackage{amssymb}
\usepackage{amsthm}
\usepackage{pdfsync}
\usepackage{esint}
\newtheorem{lemma}{Lemma}[section]
\newtheorem{theorem}{Theorem}[section]

\theoremstyle{remark}
\newtheorem{remark}{Remark}[section]
\DeclareMathOperator{\grad}{grad}
\DeclareMathOperator{\Hess}{Hess}

\numberwithin{equation}{section}

\begin{document}
\title{Apéry Polynomials and the multivariate Saddle Point Method}
\author{Thorsten Neuschel%
  \thanks{Department of Mathematics, KU Leuven, Celestijnenlaan 200B box 2400, BE-3001 Leuven, Belgium. This work is supported by KU Leuven research grant OT\slash12\slash073. E-mail: Thorsten.Neuschel@wis.kuleuven.be}}

 \date{\today}

\maketitle

\paragraph{Abstract}
The Apéry polynomials and in particular their asymptotic behavior play an essential role in the understanding of the irrationality of \(\zeta(3)\). In this paper, we present a method to study the asymptotic behavior of the sequence of Apéry polynomials \((B_{n})_{n=1}^{\infty}\) in the whole complex plane as \(n\rightarrow \infty\). The proofs are based on a multivariate version of the complex saddle point method. Moreover, the asymptotic zero distributions for the polynomials \((B_{n})_{n=1}^{\infty}\) and for some transformed Apéry polynomials are derived by means of the theory of logarithmic potentials with external fields, establishing a characterization as the unique solution of a weighted equilibrium problem. The method applied is a general one, so that the treatment can serve as a model for the study of objects related to the Apéry polynomials.


\section{Introduction}

In his famous proof of the irrationality of \(\zeta(3)\), Apéry (see \cite{Apery}) uses the asymptotic behavior of the binomial sums
\[b_n = \sum_{k=0}^n \binom{n}{k}^2 \binom{n+k}{k}^2.\]
Using the three term polynomial recurrence
\begin{equation}\label{Rec}n^3 b_n - (34 n^3 -51 n^2 +27 n -5) b_{n-1} +(n-1)^3 b_{n-2} =0,
\end{equation}
it is easy to see that 
\[\lim_{n\rightarrow \infty} \frac{b_{n+1}}{b_n} = (1+\sqrt{2})^4,\]
showing that the (nowadays called) Apéry numbers \(b_n\) cannot satisfy a two term recurrence. It is attributed to Cohen (see \cite{Poorten}, but see also Hirschhorn \cite{Hirschhorn}) that we know more precisely 
\begin{equation}\label{Cohen}b_n \sim \frac{(1+\sqrt{2})^{4n+2}}{(2 \pi n \sqrt{2})^{3/2}},
\end{equation}
as \(n \rightarrow \infty\), meaning that the quotient of both sides converges to unity as \(n\) tends to infinity. McIntosh generalized this (see \cite{McIntosh}), giving complete asymptotic expansions for binomial sums of the general form
\[\sum_{k=0}^n \binom{n}{k}^{r_0} \binom{n+k}{k}^{r_1}\cdot\ldots\cdot \binom{n+mk}{k}^{r_m},\]
where \(m\) and \(r_0,\ldots,r_m\) are fixed nonnegative integers. 

Since their introduction, the Apéry numbers have been studied extensively, mostly from the perspective of number theory. Congruences and arithmetical properties have been investigated, for example, by Beukers (see \cite{Beukers2}) and by Garaev, Luca and Shparlinski (see \cite{Garaev}, \cite{Luca}). Convexity properties have been explored by Chen and Xia (see \cite{Chen}), and the recurrence relation (\ref{Rec}) has been generalized by Askey and Wilson (see \cite{Askey}).

However, the binomial sums \(b_n\) do not appear haphazardly in Apéry's proof. Studied from the perspective of Hermite-Padé approximation they emerge as evaluations of certain polynomials, which arise quite naturally in this context. In fact, the proof of the irrationality of \(\zeta(3)\) fits into the theory of Hermite-Padé approximation the following way (see Beukers \cite{Beukers1} and Van Assche \cite{Assche1}, \cite{Assche2}):
Let three Markov functions be defined by 
\[f_1 (z) =\int\limits_0^1\frac{dx}{z-x},\quad f_2 (z) =-\int\limits_0^1 \log x \frac{dx}{z-x},\quad f_3 (z) =\frac{1}{2}\int\limits_0^1 \log^2 x \frac{dx}{z-x},\]
so that we have \(f_3 (1)=\zeta(3)\). Now, we consider the following approximation problem: Find polynomials \(A_n , B_n, C_n\) and \(D_n\) of degree at most \(n\) such that
\begin{align*}
A_n (z) &= \mathcal{O}(z-1),\quad z\rightarrow 1\\
A_n (z)f_1 (z) + B_n(z) f_2 (z) - C_n(z) &= \mathcal{O}(z^{-n-1}),\quad z\rightarrow \infty\\
A_n (z)f_2 (z) + 2 B_n(z) f_3 (z) - D_n(z) &= \mathcal{O}(z^{-n-1}),\quad z\rightarrow \infty.
\end{align*}
It turns out (see \cite[p. 222]{Assche1}), that the polynomials \(B_n\) in the solution of this approximation problem are given by the nowadays called Apéry polynomials
\begin{equation}\label{AP}B_n (z)=\sum_{k=0}^n \binom{n}{k}^2 \binom{n+k}{k}^2 z^k,
\end{equation}
and that \(2B_n(1)f_3 (1)-D_n(1) = 2b_n\zeta(3)-D_n(1)\) gives a rational approximation to \(\zeta(3)\), which is suitable to prove irrationality.

According to these considerations it is natural to investigate the Apéry polynomials in (\ref{AP}) in view of their number theoretic properties, as well as in view of their asymptotic behavior. In the former sense, congruences und arithmetic properties have been studied by Sun (see \cite{Sun}). 

In this paper we are concerned with the asymptotic behavior of the Apéry polynomials in the complex plane. To this end, we make use of a multivariate version of the complex saddle point method in order to study the behavior of \(r\)-fold complex integrals of the form
\[\int\limits_{[-a,a]^r} e^{-n p(t)} q(t) dt,\]
where \(t=(t_1,\ldots,t_r)\), as \(n\rightarrow \infty\). Here, \(p\) and \(q\) are holomorphic functions on a domain of \(\mathbb{C}^r\) containing \([-a,a]^r\). In many cases, this version is suitable to study asymptotics of multivariate complex contour integrals of the form
\[\int\limits_{\Gamma} e^{-n p(t)} q(t) dt,\]
where \(\Gamma\) is an \(r\)-fold product of contours in the complex plane (see also Remark \ref{RE2}). In our applications the quantity \(n\) will be an integer, as we are dealing with polynomials, but this is not of importance for the method itself. A generalization of the saddle point method from the one-dimensional case to the multi-dimensional case is not new. Although a rigorous treatment apparently did not find its way into the present anglophone standard literature on asymptotics, a brief sketch of the general method can be found in the Russian translation \cite[p. 124]{Gamkrelidze}. However, for the proof the reader is referred to a Russian work by Fedoryuk (see \cite{Fedoryuk}), which may be difficult to access. A short description of the method without carrying out the proofs in detail can also be found in \cite[p. 18]{Zinn}. For this reason and for the purpose of self-containment of this exposition, we will provide in Section 2 a sufficient general version of the required method (Theorem \ref{MSP}) together with a convenient short rigorous proof.

Preparing for the proofs of the main results, we will state in Section 3 some auxiliary results. The first one, Lemma \ref{Z}, will be a simple preliminary localization of the zeros of the Apéry polynomials, showing that all zeros are real and negative. After that, in Lemma \ref{EST}, we will deal with an elementary inequality for the modulus of a specific function of several complex variables on the boundary of the unit disk, which will be important for the subsequent analysis of the Apéry polynomials in Theorem \ref{MSP}. The fourth and final section contains the main results, which give an asymptotic description of the sequence of the Apéry polynomials in the whole complex plane. In Theorem \ref{MAIN1}, we show that we have for \(z\) located in the zero-free domain \(\mathbb{C}\backslash(-\infty,0]\)

\[
B_{n} (z) \sim \frac{(1+\sqrt{z}+\sqrt{z+\sqrt{z}})^2}{2 (2\pi n \sqrt{z+\sqrt{z}})^{3/2}}\,\left(1+2\sqrt{z}+2\sqrt{z+\sqrt{z}}\right)^{2n}~,
\]
as \(n\rightarrow \infty\), where \(\sqrt{z} = \exp\left\{\frac{1}{2} (\log|z| +i\arg(z))\right\}\) with \(\arg(z)\in (-\pi, \pi)\). In particular, for \(z=1\) we regain the result (\ref{Cohen}). On the other hand, on the region of zeros, that is on the negative real axis, we obtain in Theorem \ref{MAIN2}

\begin{align*}
B_n (-x) =& \frac{1}{(2\pi n)^{3/2}}\,\frac{(1+\sin(\frac{\pi}{2}\theta))(1+\frac{1}{2}\tan^2(\frac{\pi}{2}\theta))}{(\frac{1}{2}\tan(\frac{\pi}{2}\theta) \sqrt{1+\cos^2(\frac{\pi}{2}\theta)})^{3/2}} \, \left\{\frac{1+\sin(\frac{\pi}{2}\theta)}{\cos(\frac{\pi}{2}\theta)}\right\}^{2n} \\\nonumber
& \times \left(\cos\left\{f(\theta)+n\pi\theta\right\}+ o(1)\right),
\end{align*}
as \(n\rightarrow \infty\), where the phase shift \(f(\theta)\) is given by
\[f(\theta)=2\arctan\left(\frac{\tan(\frac{\pi}{2}\theta) (1+\sin(\frac{\pi}{2}\theta))}{2+\sin(\frac{\pi}{2}\theta)}\right)-\frac{3}{2}\arctan\left(\frac{1}{\cos(\frac{\pi}{2}\theta)}\right)\]
and \(x>0\) is suitably parameterized by
\[x=\frac{\sin^4 (\frac{\pi}{2}\theta)}{4 \cos^2(\frac{\pi}{2}\theta)},~~~~~0<\theta<1.\]
We will not be concerned with the uniformity of these approximations on compact subsets of the range of their validity, as this will not be necessary for the study of the zeros. To characterize the behavior of the zeros of the polynomials \(B_n\), it turns out to be convenient to translate the problem to the compact interval \([-1,1]\). Thus, in the first instance, we will consider the polynomials
\[\tilde{B}_n(z)=(z+1)^n B_n \left(\frac{z-1}{z+1}\right)=\sum_{k=0}^n \binom{n}{k}^2 \binom{n+k}{k}^2 (z-1)^k (z+1)^{n-k}.\]
In the first part of Theorem 4.3, we show that the sequence of the normalized zero counting measures \((\nu_n)_n\) associated with \((\tilde{B}_n)_n\) converges in the weak-star sense to a unit measure \(\nu\) supported on \([-1,1]\). This measure can be characterized as the unique solution for the weighted equilibrium problem on \([-1,1]\) with respect to the weight function 
\[w(x)={\left|\sqrt{x+1}+2\sqrt{x-1}+2\sqrt{x-1+\sqrt{x-1}\sqrt{x+1}}\right|}^{-2},\] 
where all roots are defined by \(\sqrt{z} = \exp\left(\frac{1}{2}(\log|z|+i \arg(z))\right)\) with \(\arg(z)\in(-\pi,\pi]\). Moreover, the measure \(\nu\) is absolutely continuous with respect to the Lebesgue measure on \([-1,1]\) and with Radon-Nikodym derivative given explicitly by
\[d\nu=\frac{1}{\pi \sqrt{1+x} (1-x)^{3/4} (\sqrt{2}+\sqrt{1-x})^{1/2}} ~dx.\]
In the second part of Theorem 4.3, we translate these results to the Apéry polynomials, showing that the sequence of the normalized zero counting measures \((\mu_n)_n\) associated with \((B_n)_n\) converges in the weak-star sense to a unit measure \(\mu\) supported on \((-\infty,0]\), which is absolutely continuous with respect to the Lebesgue measure and with Radon-Nikodym derivative given explicitly by
\[d\mu=\frac{1}{\sqrt{2}\pi |x|^{3/4} (1+|x|)^{1/2} (\sqrt{|x|}+\sqrt{1+|x|})^{1/2}} ~dx.\]

We emphasize that the whole method applied is a general one. Merely the form of the auxiliary section, especially Lemma \ref{EST}, is closely connected to the fact that we are treating the sequence of Apéry polynomials here. For instance in \cite{Abel}, using the arguments of the presented method in an ad hoc way, the authors have obtained asymptotic descriptions for the binomial polynomials 
\[Q_n^{(r)} (z)=\sum_{k=0}^n \binom{n}{k}^{r+1} z^k\]
on the whole complex plane, thereby generalizing classical results on the Legendre polynomials. It turns out, that many related sequences of polynomials, that is, series of generalized hypergeometric polynomials after ``proper" rescaling, admit a suitable multivariate complex integral representation that is amendable to this analysis.

Throughout the paper, all potential theoretic notions are used as defined in the book of Saff and Totik \cite{Saff/Totik}.

\section{A multivariate saddle point method}

In this section we will provide one of the main tools for our studies of the asymptotic behavior of the Apéry polynomials. As pointed out in the introduction, a generalization of the one-dimensional saddle point method (as shown, for instance, in \cite[p. 125]{Olver}) to the multivariate case is not new (see \cite{Fedoryuk}). For clarity and self-containment of this exposition, we present a version of this method which is adequate for the applications we have in mind.

\begin{theorem}\label{MSP}
Let \(p\) and \(q\) be holomorphic functions on a complex domain \(D\subset \mathbb{C}^r\) with \([-a,a]^r \subset D\) for a number \(a>0\), and let
\[I(n)=\int\limits_{[-a,a]^r} e^{-n p(t)} q(t) dt,\]
where \(t=(t_1,\ldots,t_r)\). Moreover, let \(t=0\) be a simple saddle point of the function \(p\), which means that we have for the complex gradient
\[\grad p (0) = 0\]
and for the Hessian
\[\det\Hess p (0) \neq 0.\]
Furthermore, suppose that, considered as a real-valued function on \([-a,a]^r\), \(\Re [p (w)]\) attains its minimum exactly at the point \(t=0\) with \(\det\Re\Hess p (0) \neq 0\) and \(q(0)\neq 0\). Then we have
\begin{equation}\label{A}I(n)\sim \left(\frac{2\pi}{n}\right)^{r/2} e^{-n p(0)} \frac{q(0)}{\sqrt{\det\Hess p (0)}},
\end{equation}
as \(n\rightarrow \infty\).
\end{theorem}

\begin{remark} The choice of the proper branch for the square root in (\ref{A}) can be described in general by fixing the arguments of the eigenvalues of \(\Hess p (0)\) in a correct manner (see \cite{Fedoryuk}), which is related to the Maslov index. However, in our application of Theorem \ref{MSP}, the determination of the branch will become clear from a simple argument using analytic continuation.
\end{remark}
\begin{remark}\label{RE2}
The formulation of Theorem \ref{MSP} is adapted to the applications we have in mind. It shall be emphasized that it is possible to consider a more general setting. The result can be formulated for integrals on compact smooth manifolds obtaining a complete asymptotic expansion in (\ref{A}), and the assumption \(\det\Re\Hess p (0) \neq 0\) enables us to give a short elementary and convenient proof. However, in many cases the arising integrals can be traced back to the form treated in Theorem \ref{MSP}. On the other hand, not much is known in the case of non simple saddle points or in the case of contributing boundary points (see \cite[p. 125]{Gamkrelidze} for a short discussion). Moreover, we are not concerned with questions of uniformity of the relation (\ref{A}) with respect to further parameters the integrand may depend on (see \cite{Neuschel} for an appropriate consideration of Laplace's method for one-dimensional contour integrals).
\end{remark}

\begin{proof}[Proof of Theorem \ref{MSP}]
First of all, our assumptions imply that the matrix \(\Re\Hess p (0)\) is positive definite, so that we certainly have for arbitrarily small \(\epsilon>0\)
\begin{equation}\label{MSP1}\int\limits_{[-\epsilon,\epsilon]^r} e^{-\frac{n}{2} t^T\Hess p (0)t} q(t) dt \sim \frac{q(0)}{n^{r/2}}\int\limits_{\mathbb{R}^r} e^{-\frac{1}{2} t^T\Hess p (0)t}dt=\left(\frac{2\pi}{n}\right)^{r/2}\frac{q(0)}{\sqrt{\det\Hess p (0)}},
\end{equation}
as \(n\rightarrow \infty\), where the evaluation of the latter integral for a complex matrix can be found, for instance, in \cite[p. 20]{Zinn}. Expanding the function \(p(t)\) in a multivariate Taylor series in the origin, we obtain
\[p(t)=p(0)+\grad p(0)^Tt+\frac{1}{2}t^T\Hess p(0)t+\vert\vert t\vert\vert^2 R(t),\]
where \(R(t)\rightarrow 0\) as \(t\rightarrow 0\) and \(\vert\vert \cdot \vert\vert\) denotes the Euclidian norm on \(\mathbb{C}^r\). Let \(\delta >0\) be defined by 
\[\delta=\min_{t\in\mathbb{R}^r, \vert\vert t\vert\vert=1 } t^T\Re\Hess p(0)t\]
and choose \(\epsilon >0\) such that \(\vert R(t)\vert\leq \frac{\delta}{4}\) on \([-\epsilon, \epsilon]^r\), then we obtain for some positive constant \(C\)
\begin{align*}
&\left|n^{r/2} \left\{ \int\limits_{[-\epsilon,\epsilon]^r}e^{-n (p(t)-p(0))} q(t)dt- \int\limits_{[-\epsilon,\epsilon]^r}e^{-\frac{n}{2} t^T \Hess p(0)t}q(t)dt\right\}\right|\\
&\leq C\int\limits_{\mathbb{R}^r} f_n (t)dt,
\end{align*}
where \(f_n (t) =\mathbbm{1}_{[-\epsilon n,\epsilon n]}(t)e^{-\frac{1}{2} t^T \Re\Hess p(0)t}\left\vert e^{-\vert\vert t\vert\vert^2 R(n^{-1/2}t)}-1\right\vert\).
Now, we clearly have \(f_n (t) \rightarrow 0\), as \(n\rightarrow \infty\), pointwise on \(\mathbb{R}^r\), and \(\vert f_n (t)\vert \leq 2 e^{-\frac{\delta}{4}\vert\vert t\vert\vert^2} \in L_1 (\mathbb{R}^r)\). Using Lebesgue's convergence theorem and taking (\ref{MSP1}) into account, this implies
\begin{equation}\label{MSP2}
\int\limits_{[-\epsilon,\epsilon]^r}e^{-n p(t)} q(t)dt \sim \left(\frac{2\pi}{n}\right)^{r/2} \frac{e^{-np(0)} q(0)}{\sqrt{\det\Hess p (0)}},
\end{equation}
as \(n \rightarrow \infty\). Now, using the condition that \(\Re [p (w)]\) attains its minimum exactly at the point \(t=0\), it is easy to see that the relation (\ref{MSP2}) still remains true if we extend the range of integration to \([-a,a]^r\).
\end{proof}

\section{Auxiliary results}

Our first auxiliary result in this section will be a simple preliminary localization of the zeros of the Apéry polynomials.
\begin{lemma}\label{Z}All zeros of the Apéry polynomials \(B_{n}\) are real and negative.
\end{lemma}
\begin{proof}The statement clearly is true for the polynomials
\[\sum_{k=0}^n \binom{n}{k}^{2} z^k,\]
as they can be considered as transformed Legendre polynomials. Alternatively, see \cite{Abel}, Lemma 2.2, for a more general statement. Multiplying by the factor \(z^n\) and deriving \(n\) times, we obtain the polynomials
\[n! \sum_{k=0}^n \binom{n}{k}^{2} \binom{n+k}{k} z^k,\]
for which the statement also is true, due to a repeated application of the Gauss-Lucas theorem. Now, repeating this process once again, we can verify that all zeros of \(B_n\) are real and negative.
\end{proof}

Next, we will deal with a particular estimate, which will arise naturally in the study of the asymptotic behavior of the Apéry polynomials via Theorem \ref{MSP}.
\begin{lemma}\label{EST}Let the functions \(a, b :\mathbb{C}\backslash(-\infty,0]\rightarrow\mathbb{C}\) be defined by
\[a(z)=\sqrt{z+\sqrt{z}}-\sqrt{z},\quad b(z)=\sqrt{z+\sqrt{z}}+\sqrt{z},\]
where all complex roots are confined to their principle value, that is, 
\[\sqrt{z} = \exp\left\{\frac{1}{2} (\log|z| +i\arg(z))\right\}\] 
with \(\arg(z)\in (-\pi, \pi)\).
Moreover, for fixed \(z\in\mathbb{C}\backslash(-\infty,0]\), let \(H :(\partial\mathbb{D})^3 \rightarrow \mathbb{R}\) be defined by
\[H(w_1, w_2, w_3)=\Bigg\vert\frac{(1+b(z)w_3)(1+b(z)(w_1 w_3)^{-1})}{(1-a(z)w_2)(1-a(z)w_1 w_2^{-1})}\Bigg\vert,\]
where \(\partial\mathbb{D}\) denotes the boundary of the complex unit disk.
Then the function \(H\) attains its global maximum exactly at the point \(w_1 = w_2 = w_3 =1\).
\end{lemma}
\begin{proof}For \(z\in\mathbb{C}\backslash(-\infty,0]\) and \(t_1, t_2, t_3 \in [-\pi,\pi]\), we want to show the estimate
\begin{equation}\label{3.1}
\left\vert\frac{(1+b(z)e^{i t_3})(1+b(z)e^{-i (t_1 +t_3)})}{(1-a(z)e^{i t_2})(1-a(z)e^{i (t_1 -t_2)})}\right\vert\leq \left\vert\frac{(1+b(z))^2}{(1-a(z))^2}\right\vert, 
\end{equation}
where the inequality is an equality, if and only if we have \(t_1= t_2= t_3=0\). Using the definitions of the functions \(a(z)\) and \(b(z)\), an elementary calculation shows that (\ref{3.1}) can be equivalently formulated the following way: For \(z\in\mathbb{C}\) with \(\Re(z)>0\) and \(t_1, t_2, t_3 \in [-\pi,\pi]\) we would like to know that
\begin{align}\nonumber
&\left\vert(\sqrt{1+z}+e^{i t_3}-1)(\sqrt{1+z}+e^{-i (t_1+t_3)}-1) \right\vert\\\label{3.2}
&\leq \left\vert(\sqrt{1+z}-e^{i t_2}+1)(\sqrt{1+z}-e^{i (t_1-t_2)}+1)\right\vert, 
\end{align}
where the inequality is an equality, if and only if we have \(t_1= t_2= t_3=0\). First of all, using the inequality of arithmetic and geometric means, we see for the left-hand side of (\ref{3.2})
\begin{align*}
&\left\vert(\sqrt{1+z}+e^{i t_3}-1)(\sqrt{1+z}+e^{-i (t_1+t_3)}-1) \right\vert\\
&\leq \frac{1}{2}\left\{\vert\sqrt{1+z}-1+e^{i t_3}\vert^2 +\vert\sqrt{1+z}-1+e^{-i (t_1 +t_3)}\vert^2\right\}\\
&=\vert\sqrt{1+z}-1\vert^2 +1+2\cos\left(t_3+\frac{t_1}{2}\right)\Re\{(\sqrt{1+z}-1)e^{i \frac{t_1}{2}}\}, 
\end{align*}
from which we immediately obtain
\begin{align}\nonumber
&\left\vert(\sqrt{1+z}+e^{i t_3}-1)(\sqrt{1+z}+e^{-i (t_1+t_3)}-1) \right\vert\\\label{3.3}
&\leq \max\{\vert\sqrt{1+z}-1+e^{-i \frac{t_1}{2}}\vert^2, \vert\sqrt{1+z}-1-e^{-i \frac{t_1}{2}}\vert^2\}.
\end{align}
Next, we want to prove that the right-hand side of (\ref{3.3}) is less than or equal to the right-hand side of (\ref{3.2}). To this end, we show the following estimate: For arbitrary \(\alpha\in [-\pi,\pi]\), \(z\in\mathbb{C}\) with \(\Re(z)>0\) and \(x\in[-\pi,\pi]\) we have
\begin{equation}\label{3.4}
\vert\sqrt{1+z}-1+e^{-i \alpha}\vert^2 \leq \vert\sqrt{1+z}+1-e^{i \alpha}e^{i x} \vert \vert\sqrt{1+z}+1-e^{i \alpha}e^{-i x}\vert, 
\end{equation}
where the inequality is an equality, if and only if \(x=\alpha=0\). In order to establish (\ref{3.4}), we study its right-hand side as a differentiable function of \(x\in[-\pi,\pi]\), which we shall denote for the moment by \(f(x)\). It is not difficult to see, that its derivative vanishes at a point \(x\), if and only if we have
\[\sin(x)=0 \quad \text{or}\quad (1+\vert \sqrt{1+z}+1\vert^2)\Re\{(\sqrt{1+z}+1)e^{-i\alpha}\}-2\cos(x)\vert \sqrt{1+z}+1\vert^2 =0.\]
This means, in the case
\begin{equation}\label{3.5}\vert\Re\{(\sqrt{1+z}+1)e^{-i\alpha}\}\vert< \frac{2\vert \sqrt{1+z}+1\vert^2}{1+\vert \sqrt{1+z}+1\vert^2},
\end{equation}
we have
\begin{align*}&\min_{x\in[-\pi,\pi]} f(x) = \min_{x\in[-\pi,\pi]}\vert\sqrt{1+z}+1-e^{i \alpha}e^{i x} \vert \vert\sqrt{1+z}+1-e^{i \alpha}e^{-i x}\vert\\
&=\min\{f(0), f(\pi), f(x^{\ast})\},
\end{align*}
where \(x^{\ast} \in (0,\pi)\) is determined by the condition
\[\cos(x^{\ast})=\frac{1+\vert \sqrt{1+z}+1\vert^2}{2\vert \sqrt{1+z}+1\vert^2} \Re\{(\sqrt{1+z}+1)e^{-i\alpha}\}.\]
On the other hand, in the remaining case
\[\vert\Re\{(\sqrt{1+z}+1)e^{-i\alpha}\}\vert\geq \frac{2\vert \sqrt{1+z}+1\vert^2}{1+\vert \sqrt{1+z}+1\vert^2},\]
we have
\[\min_{x\in[-\pi,\pi]} f(x) =\min\{f(0), f(\pi)\}.\]
According to these considerations, we will show three estimates. The first one will be
\begin{equation}\label{3.6}
\vert\sqrt{1+z}-1+e^{-i \alpha}\vert^2 \leq \vert\sqrt{1+z}+1-e^{i \alpha}\vert^2=f(0),
\end{equation}
where the inquality is an equality, if and only if \(\alpha=0\). Here, however, it is easy to see, that (\ref{3.6}) is equivalent to
\[\cos(\alpha)\Re\{\sqrt{1+z}\}\leq \Re\{\sqrt{1+z}\}.\]
The second estimate will be 
\begin{equation}\label{3.7}
\vert\sqrt{1+z}-1+e^{-i \alpha}\vert^2 < \vert\sqrt{1+z}+1+e^{i \alpha}\vert^2=f(\pi)
\end{equation}
and, under the condition (\ref{3.5}), the third estimate will be 
\begin{equation}\label{3.8}
\vert\sqrt{1+z}-1+e^{-i \alpha}\vert^2 < \vert\sqrt{1+z}+1-e^{i \alpha}e^{i x^{\ast}}\vert \vert\sqrt{1+z}+1-e^{i \alpha}e^{-i x^{\ast}}\vert=f(x^{\ast}).
\end{equation}
Let us first deal with the estimate (\ref{3.7}). An easy calculation shows that (\ref{3.7}) is equivalent to
\[\Re\{\sqrt{1+z}\}+\sin(\alpha)\Im\{\sqrt{1+z}\}+\cos(\alpha)>0.\]
Studying the extremal points of the function on the left-hand side, it turns out to be sufficient to know
\[(\Re\{\sqrt{1+z}\})^2>1+(\Im\{\sqrt{1+z}\})^2,\]
which certainly is true for values of \(z\) with \(\Re\{z\}>0\). In the next step, we will take care of the estimate (\ref{3.8}). Using the condition determining \(x^{\ast}\), we see the following:
\begin{align*}
&\vert\sqrt{1+z}+1-e^{i \alpha}e^{i x^{\ast}}\vert \vert\sqrt{1+z}+1-e^{i \alpha}e^{-i x^{\ast}}\vert\\
 &=\frac{\vert\sqrt{1+z}+1\vert^2-1}{\vert\sqrt{1+z}+1\vert}\vert\Im\{(\sqrt{1+z}+1)e^{-i\alpha}\}\vert\\
 &=(\vert\sqrt{1+z}+1\vert^2-1)\left(1-\left(\Re\left\{\frac{\sqrt{1+z}+1}{\vert\sqrt{1+z}+1\vert}e^{-i\alpha}\right\}\right)^2\right)^{1/2}.
\end{align*}
Now, using condition (\ref{3.5}), we obtain
\[\vert\sqrt{1+z}+1-e^{i \alpha}e^{i x^{\ast}}\vert \vert\sqrt{1+z}+1-e^{i \alpha}e^{-i x^{\ast}}\vert>\frac{(\vert\sqrt{1+z}+1\vert^2-1)^2}{1+\vert\sqrt{1+z}+1\vert^2}.\]
Moreover, applying the estimate
\begin{equation}\label{3.9}
\frac{(\vert\sqrt{1+z}+1\vert^2-1)^2}{1+\vert\sqrt{1+z}+1\vert^2} \geq\left(1+\frac{\vert z\vert}{\vert\sqrt{1+z}+1\vert}\right)^2= \left(1+\vert\sqrt{1+z}-1\vert\right)^2, 
\end{equation}
we readily arrive at (\ref{3.8}). Hence, in order to finish the proof, we finally have to verify the inequality in (\ref{3.9}) for values of \(z\) with \(\Re\{z\}>0\). Therefore, for any \(c>2\), we parameterize a part of the level set \(\vert\sqrt{1+z}+1\vert=c\), ``covering" the right half-plane, using
\[\gamma_c (t)=r_c (t)e^{i t}-1,\quad t\in\left[-\frac{\pi}{2},\frac{\pi}{2}\right],\]
where \(r_c (t) = \left(\sqrt{\cos(\frac{t}{2})^2-1+c^2}-\cos(\frac{t}{2})\right)^2\). Now, observing
\[\vert\gamma_c (t)\vert^2 =(r_c (t)-1)^2 +2(1-\cos(t))r_c (t),\]
we see that \(\vert\gamma_c (t)\vert\) is a decreasing function on \(\left[-\frac{\pi}{2},0\right]\), whereas it is increasing on \(\left[0,\frac{\pi}{2}\right]\). Hence, choosing an arbitrary \(z\) with \(\Re\{z\}>0\) and ``travelling" along its level set towards the imaginary axis (in the positive direction, if \(\Im\{z\}\geq0\), and in the negative direction, if \(\Im\{z\}<0\)), the inequality (\ref{3.9}) becomes sharper. Thus, it is sufficient to verify its validity only for purely imaginary values of \(z\). For convenience, let us introduce
\[g(x)=\sqrt{1+\sqrt{1+x^2}},\]
which is a strictly increasing function on \([0,\infty)\). Then, for values \(z=ix\) with \(x\geq 0\), the inequality in (\ref{3.9}) can be written as
\[\left(g(x)\left(g(x)+\sqrt{2}\right)-1\right)^2\geq \left(1+g(x)\left(g(x)+\sqrt{2}\right)\right)\left(1+\sqrt{g(x)\left(g(x)-\sqrt{2}\right)}\right)^2.\]
Hence, it is sufficient to show the inequality
\[\left(t(t+\sqrt{2})-1\right)^2\geq (1+t(t+\sqrt{2}))\left(1+\sqrt{t(t-\sqrt{2})}\right)^2, \quad t\geq \sqrt{2},\]
which, by an elementary calculation, can be reduced to the polynomial inequality
\[t^5-\sqrt{2}t^4-(4-2\sqrt{2})t^3+5t+\sqrt{2}\geq 0,\quad t\geq \sqrt{2}.\]
This inequality can be verified easily. For reasons of symmetry, this establishes (\ref{3.9}) for all purely imaginary values of \(z\), which provides the last step in the proof of the estimate in (\ref{3.1}). Finally, from the fact that we can give exact conditions for both inequalities (\ref{3.3}) and (\ref{3.4}) becoming equalities, we can deduce that the inequality in (\ref{3.1}) is an equality, if and only if \(t_1=t_2=t_3=0\).
\end{proof}

\section{Main Results}
Now we turn to the first main result which deals with the asymptotic behavior of the Apéry polynomials on the zero-free region \(\mathbb{C}\backslash(-\infty, 0]\).

\begin{theorem}\label{MAIN1}
For every complex \(z\in\mathbb{C}\backslash(-\infty,0]\) we have
\begin{equation}\label{ASYMP1}
B_{n} (z) \sim \frac{(1+\sqrt{z}+\sqrt{z+\sqrt{z}})^2}{2 (2\pi n \sqrt{z+\sqrt{z}})^{3/2}}\,\left(1+2\sqrt{z}+2\sqrt{z+\sqrt{z}}\right)^{2n}~,
\end{equation}
as \(n\rightarrow \infty\), where \(\sqrt{z} = \exp\left\{\frac{1}{2} (\log|z| +i\arg(z))\right\}\) with \(\arg(z)\in (-\pi, \pi)\). In particular
\[B_{n} (1) \sim \frac{(1+\sqrt{2})^2}{(2\pi \sqrt{2}n)^{3/2}} (1+\sqrt{2})^{4n},~~~~~~~\text{as}~~n\rightarrow\infty.\]
\end{theorem}
\begin{proof}In the first instance, we want to provide a suitable integral representation for the Apéry polynomials (\ref{AP}). To this end, we consider the polynomials \(B_n\) to be constructed as an iterated Hadamard product of the Maclaurin series for the functions \((1+z)^n\) and \((1-z)^{-n-1}\). From a repeated application of the contour integral representation for Hadamard products, we obtain the multivariate complex contour integral representation
\[B_{n} (z)=\frac{1}{(2\pi i)^3} \int\limits_{\gamma_\delta}\int\limits_{\gamma_{\delta'}}\int\limits_{\gamma_\epsilon}\left\{\frac{(1+w_3)(1+\frac{z}{w_1 w_3})}{(1-w_2)(1-\frac{w_1}{w_2})}\right\}^n \,\frac{dw_3 dw_2 dw_1}{(1-w_2)(1-\frac{w_1}{w_2}) w_1 w_2 w_3},\]
\(0<\delta<\delta'<1\), \(\epsilon>0\), where we can choose the paths of integration as positively oriented circles around the origin with the radii \(\delta, \delta'\) and \(\epsilon\). Now, recalling that \(z\in\mathbb{C}\backslash(-\infty,0]\) and studying the function
\[\frac{(1+w_3)(1+\frac{z}{w_1 w_3})}{(1-w_2)(1-\frac{w_1}{w_2})},\]
we find that its saddle points are located at
\[\left(\left(\sqrt{z+\sqrt{z}}-\sqrt{z}\right)^2,\sqrt{z+\sqrt{z}}-\sqrt{z},\sqrt{z+\sqrt{z}}+\sqrt{z}\right),\]
where we will choose the one associated to the principle branches of the roots. Thus, we are lead to introduce the particular parameterizations
\begin{align*}
w_1 =&a(z)^2 e^{i t_1},~~~t_1 \in [-\pi,\pi],\\
w_2 =&a(z) e^{i t_2},~~~t_2 \in [-\pi,\pi],\\
w_3 =&b(z) e^{i t_3},~~~t_3 \in [-\pi,\pi],
\end{align*}
where the functions \(a(z)\) and \(b(z)\) are defined as in Lemma \ref{EST}. From this, we immediately obtain the representation
\begin{equation}\label{4.2}B_{n} (z)=\frac{1}{(2\pi)^3}\int\limits_{[-\pi,\pi]^3}\left\{\frac{(1+b(z) e^{i t_3})(1+b(z) e^{-i (t_1 +t_3})}{(1-a(z) e^{i t_2})(1-a(z) e^{i (t_1 -t_2)})}\right\}^n \frac{d(t_1, t_2, t_3)}{(1-a(z) e^{i t_2})(1-a(z) e^{i (t_1 -t_2)})}.
\end{equation}
We will write the integral in (\ref{4.2}) as
\[B_{n} (z)=\frac{1}{(2\pi)^3}\int\limits_{[-\pi,\pi]^3} e^{-n p(t)} q(t)\,dt,\]
where \(t=(t_1, t_2, t_3)\) and
\begin{align*}
&p(t)=\log\left(1-a(z) e^{i t_2}\right)+ \log\left(1-a(z) e^{i (t_1 -t_2)}\right)\\
&\quad\quad-\log\left(1+b(z) e^{i t_3}\right)-\log\left(1+b(z) e^{-i (t_1 +t_3)}\right),\\
&q(t)=\frac{1}{(1-a(z) e^{i t_2})(1-a(z) e^{i (t_1 -t_2)})}.
\end{align*}
Now, from Lemma \ref{EST} we know that \(\Re\{p(t)\}\) attains its global minimum exactly at the origin. Calculating the complex partial derivatives gives us
\[\grad p (0) =0\]
and
\[\Hess p (0)=\frac{1}{1+\sqrt{z}}\begin{pmatrix} 
 2z^{1/4}\sqrt{1+\sqrt{z}} & -\sqrt{z+\sqrt{z}}-\sqrt{z} & \sqrt{z+\sqrt{z}}-\sqrt{z}  \\
 -\sqrt{z+\sqrt{z}}-\sqrt{z}  & 2\sqrt{z+\sqrt{z}}+2\sqrt{z}  & 0 \\
 \sqrt{z+\sqrt{z}}-\sqrt{z} & 0 & 2\sqrt{z+\sqrt{z}}-2\sqrt{z} 
\end{pmatrix},\]
from which we readily obtain
\[\det\Hess p(0)=\frac{4 z^{3/4}}{(1+\sqrt{z})^{5/2}}.\]
Moreover, we observe that the matrix \(\Hess p (0)\) is of the form
\[\begin{pmatrix} 
 u+v & -v & u  \\
 -v  & 2v  & 0 \\
 u & 0 & 2u
\end{pmatrix},\]
where \(u\) and \(v\) are given by \(u=\frac{\sqrt{z+\sqrt{z}}-\sqrt{z}}{1+\sqrt{z}}\) and \(v=\frac{\sqrt{z+\sqrt{z}}+\sqrt{z}}{1+\sqrt{z}}\). It is easy to see that the real part of such a matrix always is positive definite, provided that we have \(\Re u>0\) and \(\Re v>0\). As these conditions are satisfied, we can conclude that \(\Re\Hess p(0)\) is positive definite. Hence, all conditions of Theorem \ref{MSP} are satisfied and we can obtain the asymptotic form (\ref{ASYMP1}) now from an application of (\ref{A}). Here, the choice of the branch of the square root in (\ref{A}), that is, the evaluation of the integral
\[\int\limits_{\mathbb{R}^3} e^{-\frac{1}{2} t^T\Hess p (0)t}dt=\frac{(2\pi)^{3/2}}{\sqrt{\det\Hess p (0)}},\]
can be established in a simple way by performing it for positive values of \(z\) first, so that the matrix \(\Hess p (0)\) is positive definite, and after that it can be extended by analytic continuation.
\end{proof}

Next we will study the behavior of the Apéry polynomials on the negative real axis expecting an oscillatory approximation in view of Lemma \ref{Z}.

\begin{theorem}\label{MAIN2}
For every \(x >0\) we have
\begin{align}\label{4.3}
B_n (-x) =& \frac{1}{(2\pi n)^{3/2}}\,\frac{(1+\sin(\frac{\pi}{2}\theta))(1+\frac{1}{2}\tan^2(\frac{\pi}{2}\theta))}{(\frac{1}{2}\tan(\frac{\pi}{2}\theta) \sqrt{1+\cos^2(\frac{\pi}{2}\theta)})^{3/2}} \, \left\{\frac{1+\sin(\frac{\pi}{2}\theta)}{\cos(\frac{\pi}{2}\theta)}\right\}^{2n} \\\nonumber
& \times \left(\cos\left\{f(\theta)+n\pi\theta\right\}+ o(1)\right),
\end{align}
as \(n\rightarrow \infty\), where the phase shift \(f(\theta)\) explicitly is given by
\[f(\theta)=2\arctan\left(\frac{\tan(\frac{\pi}{2}\theta) (1+\sin(\frac{\pi}{2}\theta))}{2+\sin(\frac{\pi}{2}\theta)}\right)-\frac{3}{2}\arctan\left(\frac{1}{\cos(\frac{\pi}{2}\theta)}\right)\]

and \(x>0\) is parameterized by
\begin{equation}\label{4.4}x=\frac{\sin^4 (\frac{\pi}{2}\theta)}{4 \cos^2(\frac{\pi}{2}\theta)},~~~~~0<\theta<1.
\end{equation}
\end{theorem}
\begin{proof}
First we observe, that the integral representation (\ref{4.2}) remains valid, if we move the argument \(z\) from \(\mathbb{C}\backslash(-\infty,0]\) onto the negative real axis. For this purpose, we stipulate the functions \(a\) and \(b\) (see Lemma \ref{EST} for the definitions) to be continuous on the cut \((-\infty,0]\) ``from above", that is, for all \(x>0\) we have
\[a(-x)=\sqrt{-x+i\sqrt{x}}-i\sqrt{x},\quad b(-x)=\sqrt{-x+i\sqrt{x}}+i\sqrt{x},\]
where all roots take their principle value on \(\mathbb{C}\backslash(-\infty,0]\). Hence, if we have \(x>0\), we obtain for \(B_{n} (-x)\) the representation
\begin{equation}\label{4.5}\frac{1}{(2\pi)^3}\int\limits_{[-\pi,\pi]^3}\left\{\frac{(1+b(-x) e^{i t_3})(1+b(-x) e^{-i (t_1 +t_3})}{(1-a(-x) e^{i t_2})(1-a(-x) e^{i (t_1 -t_2)})}\right\}^n \frac{d(t_1, t_2, t_3)}{(1-a(-x) e^{i t_2})(1-a(-x) e^{i (t_1 -t_2)})}.
\end{equation}
In the case of \(z\in \mathbb{C}\backslash(-\infty,0]\), the proof of Theorem \ref{ASYMP1} shows, that the asymptotic behavior of \(B_n(z)\) is determined by a single saddle point located at 
\[\left(a(z)^2, a(z), b(z)\right)=\left(\left(\sqrt{z+\sqrt{z}}-\sqrt{z}\right)^2,\sqrt{z+\sqrt{z}}-\sqrt{z},\sqrt{z+\sqrt{z}}+\sqrt{z}\right).\]
Moving \(z\) onto the negative axis, let us say to \(z=-x\), we can see, that the point
\begin{equation}\label{4.6}\left(a(-x)^2, a(-x), b(-x)\right)
\end{equation}
still gives a saddle point for the function
\[\frac{(1+w_3)(1+\frac{-x}{w_1 w_3})}{(1-w_2)(1-\frac{w_1}{w_2})},\]
but this is not the only one located in the range of integration in (\ref{4.5}) anymore. Now, as complex conjugation coincides with a change of the branches in (\ref{4.6}), we see that a further saddle point is located at
\[\left(\overline{a(-x)}^2, \overline{a(-x)}, \overline{b(-x)}\right).\]
Moreover, a careful consideration of the arguments used in the proof of Lemma \ref{EST} shows that the inequality (\ref{3.7}) may become an equality now. As a consequence it turns out that the modulus of the function
\[\frac{(1+b(-x) e^{i t_3})(1+b(-x) e^{-i (t_1 +t_3})}{(1-a(-x) e^{i t_2})(1-a(-x) e^{i (t_1 -t_2)})}\]
attains its maximum not only in the origin, which corresponds with the original saddle point, but also in the second saddle point, that is at
\begin{equation}\label{4.7}t=(-4\arg(a(-x)), -2\arg(a(-x)), -2\arg(b(-x))).
\end{equation}
Hence, in order to study the asymptotic form of the integral (\ref{4.5}) via Theorem \ref{MSP}, we have to split the range of integration into two disjoint parts, say \([-\pi,\pi]^3= E_1 \cup E_2\), where the set \(E_1\) contains a neighborhood of the origin and \(E_2\) contains a neighborhood of the point (\ref{4.7}). Thus, we obtain
\[B_{n} (-x)= I_n^{(1)}(-x) + I_n^{(2)}(-x),\]  
where we define
\[I_n^{(1)}(-x) = \frac{1}{(2\pi)^3}\int\limits_{E_1} e^{-n p(t)} q(t)\,dt, \quad I_n^{(2)}(-x)=\frac{1}{(2\pi)^3}\int\limits_{E_2} e^{-n p(t)} q(t)\,dt,\]
with, in view of (\ref{4.5}), obvious definitions of the functions \(p\) and \(q\). Now, after relocating the saddle point in the second integral to the origin, both of the integrals \(I_n^{(1)}(-x)\), \(I_n^{(2)}(-x)\) can be evaluated asymptotically using Theorem \ref{MSP}, and their asymptotic contributions will be complex conjugates (which, in fact, causes the oscillatory nature of their sum). Hence, if we introduce
\[G_n(-x)=\frac{(1+i\sqrt{x}+\sqrt{-x+i\sqrt{x}})^2}{2 (2\pi n \sqrt{-x+i\sqrt{x}})^{3/2}}\,\left(1+2i\sqrt{x}+2\sqrt{-x+i\sqrt{x}}\right)^{2n},\]
then we have
\[I_n^{(1)}(-x) = G_n(-x)\left(1+o(1)\right),\]
and
\[I_n^{(2)}(-x) = \overline{G_n(-x)}\left(1+o(1)\right),\]
as \(n\rightarrow \infty\). Thus, for large values of \(n\) we obtain
\[B_{n} (-x)=G_n(-x)\left(1+o(1)\right) + \overline{G_n(-x)}\left(1+o(1)\right).\]
Now, using the parameterization for \(x>0\) given in (\ref{4.4}), it is merely a matter of calculation to find (\ref{4.3}).
\end{proof}

\begin{remark}The parameterization (\ref{4.4}) does not appear from nowhere, in fact, it is connected to the asymptotic distribution of the zeros of the polynomials \(B_n\). In this and in related cases, a candidate for a suitable parameterization can be obtained by means of the inverse function of the limiting distribution of the zeros, which we will derive in Theorem \ref{MAIN3} relying only on Theorem \ref{MAIN1}. 
\end{remark}

Our next aim is to study the behavior of the zeros of the Apéry polynomials using tools from complex potential theory. As pointed out in the introduction, in the first instance, we will be concerned with the polynomials
\[\tilde{B}_n (z)=(z+1)^n B_n \left(\frac{z-1}{z+1}\right)=\sum_{k=0}^n \binom{n}{k}^2 \binom{n+k}{k}^2 (z-1)^k (z+1)^{n-k}.\]
Recalling that the complex transformation \(T (x) = \frac{x-1}{x+1}\) maps the interval \((-1, 1)\) bijectively onto the negative real axis \((-\infty, 0)\), we can conclude from Lemma \ref{Z} that all zeros of the polynomials \(\tilde{B}_n\) are contained in the interval \((-1, 1)\).

\begin{theorem}\label{MAIN3}
Let \((\nu_n)\) denote the sequence of the normalized zero counting measures associated with \(\tilde{B}_n\) and let \((\mu_n)\) be the sequence of the normalized zero counting measures associated with \(B_n\). Then we have:
\begin{itemize}
\item[(i)] The sequence \((\nu_n)_n\) converges in the weak-star sense to a unit measure \(\nu\) supported on \([-1,1]\). This limit measure \(\nu\) is the unique solution for the weighted equilibrium problem on \([-1,1]\) with respect to the weight function
\begin{equation}\label{4.8}w(x)={\left|\sqrt{x+1}+2\sqrt{x-1}+2\sqrt{x-1+\sqrt{x-1}\sqrt{x+1}}\right|}^{-2},
\end{equation}
all roots defined by \(\sqrt{z} = \exp\left(\frac{1}{2}(\log|z|+i \arg(z))\right)\) with \(\arg(z)\in(-\pi,\pi]\), and its logarithmic potential is given by
\[\mathcal{U}^{\nu}(z)=\log(1+\sqrt{2})^4-2\log\left|\sqrt{z+1}+2\sqrt{z-1}+2\sqrt{z-1+\sqrt{z-1}\sqrt{z+1}}\right|,\quad z\in\mathbb{C}.\]
Moreover, the measure \(\nu\) is absolutely continuous with respect to the Lebesgue measure on \([-1,1]\) and its Radon-Nikodym derivative is given explicitly by
\begin{equation}\label{D}d\nu=\frac{1}{\pi \sqrt{1+x} (1-x)^{3/4} (\sqrt{2}+\sqrt{1-x})^{1/2}} ~dx.
\end{equation}

\item[(ii)] The sequence \((\mu_n)_n\) converges in the weak-star sense to a unit measure \(\mu\) supported on \((-\infty,0]\) possessing the logarithmic potential
\[\mathcal{U}^{\mu}(z)=4\log2-2\log\left|1+2\left(\sqrt{z}+\sqrt{z+\sqrt{z}}\right)\right|,\quad z\in\mathbb{C},\]
where, as before, all roots are defined by \(\sqrt{z} = \exp\left(\frac{1}{2}(\log|z|+i \arg(z))\right)\) with \(\arg(z)\in(-\pi,\pi]\).
Moreover, the measure \(\mu\) is absolutely continuous with respect to the Lebesgue measure and its Radon-Nikodym derivative is given explicitly by
\[d\mu=\frac{1}{\sqrt{2}\pi |x|^{3/4} (1+|x|)^{1/2} (\sqrt{|x|}+\sqrt{1+|x|})^{1/2}} ~dx.\]
\end{itemize}
\end{theorem}
\begin{proof}
In consideration of the results we gained so far, the proof of Theorem \ref{MAIN3} can be established by adapting the arguments used in \cite{Abel} to prove similar results in Theorem 3.3 on the binomial polynomials, which also gives a hint at the general applicability of this approach.
\begin{itemize}
\item[(i)] In the first step, we use the asymptotic behavior of the Apéry polynomials \(B_n\) on the zero-free region to obtain an explicit representation for the logarithmic potential of the limit distribution of zeros of the transformed polynomials \(\tilde{B}_n\). More precisely, for complex \(z\) we have
\begin{align*}
\mathcal{U}^{\nu_n} (z)=&\int\log|z-t|^{-1}\,d\nu_n(t)\\
&=\log|B_n (1)|^{1/n} -\log\left(|z+1|\left|B_n \left(\frac{z-1}{z+1}\right)\right|^{1/n}\right).
\end{align*}
For \(z\in\mathbb{C}\backslash[-1,1]\) we therefore obtain by an application of Theorem \ref{MAIN1}
\begin{align}\label{4.9}\nonumber
\lim_{n\rightarrow\infty}\mathcal{U}^{\nu_n} (z) =&\log(1+\sqrt{2})^4\\
&-2\log\left|\sqrt{z+1}+2\left(\sqrt{z-1}+\sqrt{z-1+\sqrt{z-1}\sqrt{z+1}}\right)\right|.
\end{align}
Hence, using Helly's selection principle (\cite[p. 3]{Saff/Totik}) and Carleson's unicity theorem (\cite[p. 123]{Saff/Totik}) we can conclude that the sequence \((\nu_n)\) converges in the weak-star sense to a unit measure \(\nu\) supported on \([-1,1]\) and its logarithmic potential \(\mathcal{U}^{\nu}\) is given by the right-hand side of the equation (\ref{4.9}) on \(\mathbb{C}\backslash[-1,1]\). Now, using that the potential \(\mathcal{U}^{\nu}\) and the right-side of the equation (\ref{4.9}) is continuous on \(\mathbb{C}\) with respect to the fine topology (see, e.g. \cite[p. 58]{Saff/Totik}) and noting that the boundary of \(\mathbb{C}\backslash[-1,1]\) in the fine topology coincides with its boundary in the Euclidean topology (see \cite[p. 61]{Saff/Totik}) we obtain for all \(z\in\mathbb{C}\)
\[
\mathcal{U}^{\nu} (z)=\log(1+\sqrt{2})^4-2\log\left|\sqrt{z+1}+2\left(\sqrt{z-1}+\sqrt{z-1+\sqrt{z-1}\sqrt{z+1}}\right)\right|.\]
As the measure \(\nu\) clearly has finite logarithmic energy and as we have
\[\mathcal{U}^{\nu} (z)+2\log\left|\sqrt{z+1}+2\left(\sqrt{z-1}+\sqrt{z-1+\sqrt{z-1}\sqrt{z+1}}\right)\right|=\log(1+\sqrt{2})^4,\]
we can deduce from Remark 1.5 in \cite[p. 28]{Saff/Totik}, that \(\nu\) solves the equilibrium problem on \([-1,1]\) with respect to the weight function \(w(x)\) defined in (\ref{4.8}) and its modified Robin constant is given by \(4\log(1+\sqrt{2})\). Now, in order to prove that the Radon-Nikodym derivative of \(\nu\) is given by the expression (\ref{D}), using Carleson's unicity theorem, it is sufficient to verify for \(z\in\mathbb{C}\backslash[-1,1]\) the identity
\[\int_{-1}^1 \frac{\log\vert z-x\vert^{-1}}{\pi \sqrt{1+x} (1-x)^{3/4} (\sqrt{2}+\sqrt{1-x})^{1/2}}\,dx=\mathcal{U}^{\nu} (z),\]
which is merely a matter of integral calculus. However, this modus operandi requires to have an idea of the density function in advance. In many cases, this can be provided, for instance, by recovering locally the density from its logarithmic potential via the formula
\[-\frac{1}{2\pi}\left(\frac{\partial \mathcal{U}^{\nu}}{\partial n_{+}}+\frac{\partial \mathcal{U}^{\nu}}{\partial n_{-}}\right)\]
involving the normal derivatives of the potential (see, e.g., \cite{Saff/Totik}, Chapter 2).

\item[(ii)] Finally, we want to translate these results to the Apéry polynomials \(B_n\). Therefore, we can consider the measures \(\mu_n\) to be the images of the measures \(\nu_n\) with respect to the transformation \(T(x)=\frac{x-1}{x+1}\), i.e. \(\mu_n = \nu_n^T\). From the weak-star convergence of the sequence \((\nu_n)_n\) to \(\nu\) it follows that the sequence \((\mu_n)_n\) converges in the weak-star sense to the measure \(\mu=\nu^T\). Using this connection yields
\begin{align*}
\mathcal{U}^{\mu} (z)&=\mathcal{U}^{\nu} \left(\frac{z+1}{1-z}\right)-\mathcal{U}^{\nu}(-1)-\log|z-1|\\
&=4\log2-2\log\left|1+2\left(\sqrt{z}+\sqrt{z+\sqrt{z}}\right)\right|.
\end{align*}
Moreover, again using \(\mu=\nu^T\) the claimed Radon-Nikodym derivative for \(\mu\) can be obtained from those for \(\nu\) by an easy calculation.

\end{itemize} 
\end{proof}

\end{document}